\documentclass[10pt]{article}

\usepackage{amsmath}
\usepackage{amsfonts}
\usepackage{amsthm}

\usepackage{hyperref}

\allowdisplaybreaks[3]

\theoremstyle{plain}
\newtheorem{theorem}{Theorem}[section]
\newtheorem{lemma}[theorem]{Lemma}
\newtheorem{corollary}[theorem]{Corollary}

\theoremstyle{definition}

\newtheorem{remark}{Remark}

\numberwithin{equation}{section} 

\newcommand{\Ea}{E_{\alpha}}
\newcommand{\La}{\Lambda_{\alpha}}
\newcommand{\Fa}{\mathcal{F}_{\alpha}} 
\newcommand{\mua}{\mu_{\alpha}}

\newcommand{\I}{\mathcal{I}}
\newcommand{\am}{\mathfrak{a}}

\newcommand{\B}{\mathfrak{B}}
\newcommand{\E}{\mathfrak{E}}

\newcommand{\RR}{\mathbb{R}}
\newcommand{\ZZ}{\mathbb{Z}}

\newcommand{\CC}{\mathbb{C}}

\newcommand{\ZZO}{\mathbb{Z} \setminus \{0\}}

\newcommand{\cc}{\gamma}

\renewcommand{\Re}{\operatorname{Re}}
\renewcommand{\Im}{\operatorname{Im}}

\title{Bernoulli-Dunkl and Apostol-Euler-Dunkl polynomials
with applications to series involving zeros of Bessel functions%
\footnote{Partially supported by MTM2015-65888-C4-1-P and MTM2015-65888-C4-4-P (Ministerio de Econom\'ia y Competitividad),
FQM-262, FQM-7276 (Junta de Andaluc\'ia), E64 (Gobierno de Arag\'on) and Feder Funds (European Union).}}

\author{\'Oscar Ciaurri\textsuperscript{1}, Antonio J. Dur\'an\textsuperscript{2}, Mario P\'erez\textsuperscript{3} and Juan L. Varona\textsuperscript{1}\\[6pt]
\small\textsuperscript{1}Departamento de Matem\'aticas y Computaci\'on and CIME,
Universidad de La Rioja,\\[-2pt]
\small 26006 Logro\~no, Spain. Emails: {oscar.ciaurri@unirioja.es}, {jvarona@unirioja.es}\\[6pt]
\small\textsuperscript{2}Departamento de An\'alisis Matem\'atico and IMUS,
Universidad de Sevilla,\\[-2pt]
\small 41080 Sevilla, Spain. Email: {duran@us.es}\\[6pt]
\small\textsuperscript{3}Departamento de Matem\'aticas and IUMA,
Universidad de Zaragoza,\\[-2pt]
\small 50009 Zaragoza, Spain. Email: {mperez@unizar.es}}

\date{}

\begin{document}


\maketitle

\begin{abstract}
We introduce Bernoulli-Dunkl and Apostol-Euler-Dunkl polynomials as generalizations of Bernoulli and Apostol-Euler polynomials,
where the role of the derivative is now played by the Dunkl operator on the real line. We use them to sum a bunch of series involving the zeros of Bessel functions.


\end{abstract}

\section{Introduction}

Along the middle years of the XVIII-th century, Euler proved that
\begin{align}
\label{eq:Euler}
  \sum_{j=1}^{\infty} \frac{1}{j^{2k}}
  &= \frac{(-1)^{k-1}2^{2k-1}\pi^{2k}}{(2k)!} B_{2k}, \quad k=1,2,\dots,
  \\
  \label{eq:Euler2}
  \sum_{j=1}^{\infty} \frac{(-1)^{j}}{(2j-1)^{2k+1}}
  &= \frac{(-1)^{k+1} \pi^{2k+1}}{4(2k)!}\,E_{2k}, \quad k=0,1,2,\dots,
\end{align}
where $B_{2k}$ and $E_{2k}$ are the Bernoulli and Euler numbers, respectively. Bernoulli and Euler numbers are the particular values $B_{2k} = B_{2k}(0)$ and $E_{2k}=E_{2k}(1/2)$, where $\{B_n(x)\}_n$ and $\{E_n(x)\}_n$ are, respectively, the Bernoulli and Euler polynomials defined by the generating functions
\[
  \frac{te^{xt}}{e^t-1} = \sum_{n=0}^\infty B_n(x)\frac{t^n}{n!},
  \qquad
  \frac{2e^{xt}}{e^t+1} = \sum_{n=0}^\infty E_n(x)\frac{t^n}{n!}.
\]
The sum \eqref{eq:Euler} can be seen as an identity for the sum of the reciprocals of the zeros of $\sin(x)$ while in the sum \eqref{eq:Euler2}
the zeros of $\cos(x)$ are modified by an alternating sign.
The sine and cosine functions can be expressed in terms of Bessel functions: $\sin(x) = (\pi x/2)^{1/2} J_{1/2}(x)$
and  $\cos(x) = (\pi x/2)^{1/2}J_{-1/2}(x)$, respectively. So a natural generalization of~\eqref{eq:Euler} is to compute the series
\begin{equation}
\label{eq:pmj}
  \sum_{j} \frac{1}{s_{j,\alpha}^{2k}},
\end{equation}
where $\{s_{j,\alpha}\}_j$ are the zeros of a Bessel function. This question, which has both mathematical and physical interest, has a somehow classical flavour and goes back to lord Rayleigh in 1874 \cite{Ray} (in fact, the sum \eqref{eq:pmj} as a function of $\alpha$ is usually called the Rayleigh function). Since then, many papers have been published which study, with different approaches, these and many other series
along with identities and other properties; see, for instance, \cite{Sne, Ki, Ke, BMPS,dL} or \cite[formula 11 in \S\,5.7.33]{PBM} (and the list is by no means exhaustive).

However, identities relating \eqref{eq:pmj} or any other sum involving the zeros of Bessel functions
with some kind of ``Bernoulli'' or ``Euler numbers" seem to be unknown.
Such identities could be considered a true generalization of~\eqref{eq:Euler} and~\eqref{eq:Euler2}.

The purpose of this paper is to introduce what we have called Bernoulli-Dunkl and Apostol-Euler-Dunkl polynomials. We will use them to sum a bunch of series involving the  zeros of Bessel functions, among which are the analogous to the series \eqref{eq:Euler} and~\eqref{eq:Euler2}.

Our approach is the following. Bernoulli and Euler polynomials are particular cases of the so-called
Appell sequences. An Appell sequence $\{P_n(x)\}_{n=0}^{\infty}$ is a sequence of polynomials defined by a Taylor generating expansion
\begin{equation}
\label{eq:As}
  A(t) e^{xt} = \sum_{n=0}^{\infty} P_n(x) \frac{t^n}{n!},
\end{equation}
where $A(t)$ is a function analytic at $t=0$ with $A(0) \ne 0$.
Since the exponential function $e^x$ is invariant under the differential operator ${d}/{dx}$, it is easy to show that $P_n(x)$ is a polynomial of degree $n$ and $P'_n(x) = nP_{n-1}(x)$. Typical examples of Appell sequences are the Bernoulli and Euler polynomials above, or the probabilistic Hermite polynomials $\{\operatorname{He}_n(x)\}_{n=0}^{\infty}$.

For $\alpha\in \CC\setminus\{-1,-2,\dots \}$, we consider the entire functions
\begin{align*}
  \I_\alpha(z) &= 2^\alpha \Gamma(\alpha+1) \frac{J_\alpha(iz)}{(iz)^\alpha}, \\
  \Ea(z) &= \I_\alpha(z) + \frac{z}{2(\alpha+1)} \, \I_{\alpha+1}(z),
\end{align*}
where $J_\alpha$ is the Bessel function of order $\alpha$. A simple computation gives
\[
  \Ea(z) = \sum_{n=0}^\infty \frac{z^n}{\cc_{n,\alpha}}
\]
with
\[
\cc_{n,\alpha} =
\begin{cases}
  2^{2k}k!\,(\alpha+1)_k, & \text{if $n=2k$},\\
  2^{2k+1}k!\,(\alpha+1)_{k+1}, & \text{if $n=2k+1$}.
\end{cases}
\]
The entire function $\Ea$ is invariant under the Dunkl operator
\[
  \La f(x) = \frac{d}{dx}f(x) + \frac{2\alpha+1}{2}
  \left(\frac{f(x)-f(-x)}{x}\right).
\]
Let us note that, when $\alpha = -1/2$, we have $\cc_{n,-1/2} = n!$, $E_{-1/2}(x) = e^{x}$ and $\Lambda_{-1/2} = d/dx$.
Hence Appell sequences can be generalized by replacing $e^{xt}$ by $\Ea(xt)$ in~\eqref{eq:As}: given a function $A(t)$ analytic at $t=0$ such that $A(0) \ne 0$, we can associate to it an Appell-Dunkl sequence $\{A_n\}_{n=0}^{\infty}$ by the generating function
\begin{equation}
\label{eq:APs}
  A(t)\Ea(xt) = \sum_{n=0}^\infty A_n(x) \frac{t^n}{\cc_{n,\alpha}}.
\end{equation}
It is not difficult to check that $A_n$ is a polynomial of degree $n$ which satisfies $\La A_n=(n+(\alpha+1/2)(1-(-1)^n))A_{n-1}$.

Appell-Dunkl polynomials have already appeared in the literature as generalizations of the Hermite polynomials (see, for instance, \cite{AlS, Ros}).
But, as long as we know, no generalizations of Bernoulli or Euler polynomials have been considered using this approach.

This paper is organized as follows. In Section~\ref{sec:B-D}, we introduce our generalizations of Bernoulli and Euler polynomials.
More precisely, the Bernoulli-Dunkl polynomials $\{\B_{n,\alpha}\}_{n=0}^{\infty}$ are associated as in \eqref{eq:APs} to the generating function
\[
  A(t) = \frac{1}{\I_{\alpha+1}(t)}.
\]
On the other hand, the Apostol-Euler-Dunkl polynomials $\{\E_{n,\alpha,u}\}_{n=0}^{\infty}$ are associated as in \eqref{eq:APs} to the generating function
\[
  A(t) = \frac{u\I_{\alpha+1}(u)}{(t+u)\I_{\alpha+1}(t+u)},
\]
where the parameter $u$ is a complex number which is neither $0$ nor a zero of $\I_{\alpha+1}$.
In particular, we will show that $\B_{n,-1/2}(2x-1) = 2^nB_n(x)$ and $\E_{n,-1/2,i\pi/2}(2x-1) = 2^nE_n(x)$, where $B_n$ and $E_n$ are the Bernoulli and Euler polynomials, respectively.

In Section~\ref{sec:simplefr}, we consider some partial fraction decompositions related to the functions $\I_\alpha$. Using them and the polynomials introduced in Section~\ref{sec:B-D}, we sum in Section~\ref{sec:serieszeros} a bunch of series involving the zeros of Bessel functions, among which are the following examples. For any complex value of~$\alpha$ (except for the negative integers), we can order the zeros $s_{j,\alpha}$, $j\in \ZZ$, of the Bessel function $J_{\alpha+1}(x)/x^{\alpha+1}$ (notice that we have shifted the parameter $\alpha$ by~$1$) so that, $s_{j,\alpha} = -s_{-j,\alpha}$ and $0<\Re s_{j,\alpha}\le \Re s_{j+1,\alpha}$, $j\ge 0$ (\cite[\S\,15.41, p.~497]{Wat}).
The case $\alpha>-2$ is particularly relevant because then $s_{j,\alpha}$, $j\ge 0$, are positive numbers (\cite[\S\,15.27, p.~483]{Wat}). We then prove (Theorem~\ref{thm:sumassr})
\begin{equation}
\label{eq:sum1}
  \sum_{j=1}^\infty \frac{1}{s_{j,\alpha}^{2k}}
  = \frac{(-1)^{k+1}}{2^{2k}k!\,(\alpha+2)_{k-1}} \,\B_{2k,\alpha}(1),
\end{equation}
which can be considered a genuine generalization of~\eqref{eq:Euler}.

We also sum a variant of this series with ``alternate signs'' (Theorem~\ref{thm:sumassr}): for $\Re \alpha < 2k -3/2$,
\begin{equation}
\label{eq:sum2}
  \sum_{j=1}^\infty \frac{1}{\I_{\alpha}(is_{j,\alpha})s_{j,\alpha}^{2k}}
  = \frac{(-1)^{k+1}}{2^{2k}k!\,(\alpha+2)_{k-1}} \,\B_{2k,\alpha}(0).
\end{equation}
Indeed, in this context the factor $\I_{\alpha}(is_{j,\alpha})$, which multiplies each zero of the Bessel function, provides a generalization of the sign sequence $(-1)^j$ in \eqref{eq:Euler2} because for $\alpha=-1/2$ we have $\I_{-1/2}(is_{j,-1/2})=(-1)^j$. Using the Bernoulli-Dunkl polynomials one can explicitly compute the sum~\eqref{eq:sum2}.
For instance, for $k=1$ and $2$, one gets
\[
  \sum_{j=1}^\infty \frac{1}{\I_{\alpha}(is_{j,\alpha})s_{j,\alpha}^{2}}
  = -\frac{\alpha+1}{4(\alpha+2)},
  \qquad
  \sum_{j=1}^\infty \frac{1}{\I_{\alpha}(is_{j,\alpha})s_{j,\alpha}^{4}}
  = -\frac{(\alpha+1)(\alpha+2)}{32(\alpha+3)(\alpha+2)^2}.
\]
These series were considered in \cite{Sne} for $\alpha\ge -3/2$.

We also prove (Theorem~\ref{thm:cag2}) that, for $\Re \alpha < k -1/2$,
\[
  \sum_{j\in \ZZO} \frac{1}{\I_\alpha(is_{j,\alpha})(s_{j,\alpha}-u)^{n+1}}
  = 2(1+\alpha) \left(\frac{(-1)^n}{u^{n+1}}
  - \frac{i^n\E_{n,\alpha,iu}(0)}{u\I_{\alpha+1}(iu)\cc_n}\right),
\]
which can be considered a generalization of the Euler's sums \eqref{eq:Euler2}, because for $\alpha=-1/2$ and $u=\pi/2$, this series reduces to~\eqref{eq:Euler2}. Although we will insist again on this, throughout this paper any doubly infinite series of the form $\sum_{j\in \ZZ}$  must be understood as the principal value, that is, the limit of $\sum_{|j|\leq N}$ as $N$ goes to infinity.

Using our approach, we also sum some other relevant series involving the zeros of Bessel functions such as (see Theorem~\ref{thm:cag} and Theorem~\ref{thm:cag2})
\begin{gather*}
  \sum_{j\in \ZZ\setminus\{0,l\}} \frac{1}{(s_{j,\alpha}-s_{l,\alpha})^{k+1}},\\
  \sum_{j\in \ZZ\setminus\{0,l\}} \frac{1}{\I_\alpha(is_{j,\alpha})(s_{j,\alpha}-s_{l,\alpha})^{k+1}}
\end{gather*}
for $k\ge 0$. For instance,
\begin{align}
\label{eq:calggi}
  \sum_{j\in \ZZ\setminus\{0,l\}} \frac{1}{s_{j,\alpha}-s_{l,\alpha}}
  &= \frac{3+2\alpha}{2s_{l,\alpha}},\\
\nonumber
  \sum_{j\in \ZZ\setminus\{0,l\}} \frac{1}{(s_{j,\alpha}-s_{l,\alpha})^{2}}
  &= -\frac{(3+2\alpha)(7+2\alpha)}{12s_{l,\alpha}^2}+\frac{1}{3},\\
\nonumber
  \sum_{j\in \ZZ\setminus\{0,l\}} \frac{1}{\I_\alpha(is_{j,\alpha})(s_{j,\alpha}-s_{l,\alpha})}
  &= \frac{2(1+\alpha)}{s_{l,\alpha}} - \frac{1+2\alpha}{2s_{l,\alpha}\I_\alpha(is_{l,\alpha})},\\
\intertext{and}
\nonumber
  \sum_{j\in \ZZ\setminus\{0,l\}} \frac{1}{\I_\alpha(is_{j,\alpha})(s_{j,\alpha}-s_{l,\alpha})^{2}}
  &= \frac{(1+\alpha)(2\alpha-1)}{s_{l,\alpha}^2} \\
\nonumber
  &\qquad- \frac{1}{\I_\alpha(is_{l,\alpha})}
  \left( 1+\frac{(1+2\alpha)(2\alpha-3)}{2s_{l,\alpha}^2}\right).
\end{align}
The sum~\eqref{eq:calggi} is the Calogero series~\cite{Cal}.

For $\alpha >-1$, the Dunkl operator has associated the following orthonormal system:
\[
  e_{\alpha,j}(r) = \frac{2^{\alpha/2}\Gamma(\alpha+1)^{1/2}}
    {|\I_{\alpha}(is_{j,\alpha})|} \Ea(is_j r),
  \quad j\in \ZZO, \quad r \in (-1,1),
\]
and $e_{\alpha,0}(r) = 2^{(\alpha+1)/2}\Gamma(\alpha+2)^{1/2}$.
In the last section of this paper we find the following expansion of the Bernoulli-Dunkl polynomials with respect to this system (Theorem~\ref{thm:sfbd}):
\[
  \B_{n,\alpha}(x) = \frac{-(-i)^n\cc_n}{2^{1+\alpha/2}(\alpha+1)\Gamma(\alpha+1)^{1/2}}
  \sum_{j\in \ZZO} \frac{(-1)^j}{s_{j,\alpha}^n} \, e_{\alpha,j}(x),
\]
where the convergence is in $L^2\left( [-1,1],\tfrac{|x|^{2\alpha+1}}{2^{\alpha+1}\Gamma(\alpha+1)}\,dx\right)$.
This can be considered as a generalization of the Hurwitz expansion for Bernoulli polynomials in Fourier series
\[
  B_n(x) = -\frac{n!}{(2\pi i)^n} \sum_{j\in \ZZO}\frac{e^{2\pi i jx}}{j^n}.
\]

\section{Bernoulli-Dunkl and Apostol-Euler-Dunkl polynomials}
\label{sec:B-D}

As explained in the introduction, we will define Bernoulli-Dunkl and Apostol-Euler-Dunkl polynomials as particular cases of Appell-Dunkl polynomials.
To introduce Appell-Dunkl polynomials we need some preliminary notations.

\subsection{The Dunkl transform on the real line and the Appell-Dunkl polynomials}

For $\alpha > -1$, let $J_\alpha$ denote the Bessel function of order $\alpha$ and,
for complex values of the variable $z$, let
\[
  \I_\alpha(z)
  = 2^\alpha \Gamma(\alpha+1) \frac{J_\alpha(iz)}{(iz)^\alpha}
  = \Gamma(\alpha+1) \sum_{n=0}^\infty \frac{(z/2)^{2n}}{n!\,\Gamma(n+\alpha+1)}
  = {}_0F_1(\alpha+1,z^2/4)
\]
(the function $\I_\alpha$ is a slight variation of the so-called modified Bessel function of the first kind and order~$\alpha$, usually denoted by~$I_\alpha$; see~\cite{Wat}, \cite{Le} or~\cite{OlMa-NIST}).
Moreover, let us take
\[
  \Ea(z) = \I_\alpha(z) + \frac{z}{2(\alpha+1)}\,
  \I_{\alpha+1}(z),
  \qquad z \in \CC.
\]

The Dunkl operator $\La$ in the real line (with reflection group~$\ZZ_2$) is defined as
\begin{equation}
\label{eq:Duop}
  \La f(x) = \frac{d}{dx}f(x) + \frac{2\alpha+1}{2}
  \left(\frac{f(x)-f(-x)}{x}\right),
\end{equation}
acting on suitable functions $f$ on~$\RR$ (see \cite{Du} for $\alpha\ge-1/2$ and \cite{Ros} for the extension to $\alpha>-1$).
It is easy to check that, for any $\lambda \in \CC$, we have
\[
  \La \Ea(\lambda x) = \lambda \Ea(\lambda x).
\]
Let us note that $\Lambda_{-1/2} = d/dx$ and $E_{-1/2}(\lambda x) = e^{\lambda x}$.

The function $\Ea(z)$ is known as the Dunkl kernel because, in a similar way to the Fourier transform (which is the particular case $\alpha = -1/2$), we can define the Dunkl transform on the real line
\begin{equation}
\label{eq:trDu}
  \Fa f(y)
  = \int_{\RR} \Ea(-ixy)f(x)\,d\mua(x),
  \qquad y \in \RR,
\end{equation}
were $d\mua$ denotes the measure
\[
  d\mua(x) = \tfrac{1}{2^{\alpha+1}\Gamma(\alpha+1)}\, |x|^{2\alpha+1}\,dx
\]
(in particular, $d\mu_{-1/2}(x) = (2\pi)^{-1/2}\,dx$).
The Dunkl transform~\eqref{eq:trDu} can be extended to an isometric isomorphism on $L^2(\RR,\mua)$ and fulfills $\Fa^{-1}f(y) = \Fa f(-y)$.

From the definition, it is easy to check that
\[
  \Ea(z) = \sum_{n=0}^\infty \frac{z^n}{\cc_{n,\alpha}}
\]
with
\begin{equation}
\label{eq:ccna}
  \cc_{n,\alpha} =
\begin{cases}
  2^{2k}k!\,(\alpha+1)_k, & \text{if $n=2k$},\\
  2^{2k+1}k!\,(\alpha+1)_{k+1}, & \text{if $n=2k+1$},
\end{cases}
\end{equation}
and where $(a)_n$ denotes the Pochhammer symbol
\[
  (a)_n = a(a+1)(a+2) \cdots (a+n-1) = \frac{\Gamma(a+n)}{\Gamma(a)}
\]
(with $n$ a non-negative integer). Notice that $\cc_{n,-1/2} = n!$.
From~\eqref{eq:ccna}, one straightforwardly has
\begin{equation}
\label{eq:coccc}
  \frac{\cc_{n,\alpha}}{\cc_{n-1,\alpha}} = n+(\alpha+1/2)(1-(-1)^n).
\end{equation}
We also define
\[
  \binom{n}{j}_{\!\alpha} = \frac{\cc_{n,\alpha}}{\cc_{j,\alpha}\cc_{n-j,\alpha}},
\]
that becomes the ordinary binomial numbers in the case $\alpha=-1/2$.
To simplify the notation we sometimes write $\cc_{n,\alpha}=\cc_n$.

To every function $A(t)$ analytic in a neighborhood of $t=0$ with $A(0) \ne 0$, we associate a sequence of Appell-Dunkl polynomials $A_n(x)$ by the generating function
\begin{equation}
\label{eq:ad}
  A(t)\Ea(xt) = \sum_{n=0}^\infty A_n(x) \frac{t^n}{\cc_n}
\end{equation}
(in addition to the papers \cite{AlS, Ros} cited in the introduction, Appell-Dunkl sequences have been considered also, for instance, in \cite{BG, BKI, DH}).
From this definition, it is not difficult to prove that $A_n(x)$ is a polynomial of degree $n$ and $\La A_n(x) = \frac{\cc_{n}}{\cc_{n-1}}\, A_{n-1}(x)$ (when $\alpha=-1/2$, this becomes the classical $A'_n(x) = nA_{n-1}(x)$ of Appell sequences).

We straightforwardly have the following:

\begin{lemma}
\label{lem:xn}
The Appell-Dunkl polynomials $A_n(x)$, $n\ge 0$, defined by~\eqref{eq:ad} satisfy the recurrence relations
\[
  x^n = \cc_n\sum_{j=0}^n \frac{A_j(x)}{\cc_j} \, a_{n-j},
\]
where
\[
  \frac{1}{A(t)} = \sum_{n=0}^\infty a_nt^n.
\]
Moreover,
\[
  \La(A_n) = (n+(\alpha+1/2)(1-(-1)^n))A_{n-1}.
\]
\end{lemma}

\subsection{Bernoulli-Dunkl polynomials}

We define the Bernoulli-Dunkl polynomials $\{\B_{n,\alpha}\}_{n=0}^{\infty}$ by the generating function
\begin{equation}
\label{eq:defbd}
  \frac{\Ea(xt)}{\I_{\alpha+1}(t)}
  = \sum_{n=0}^\infty \frac{\B_{n,\alpha}(x)}{\cc_{n,\alpha}}\,t^n.
\end{equation}
To simplify the notation we sometimes write $\B_n=\B_{n,\alpha}$.

Since
\[
  \I_{\alpha+1}(x) = \sum_{n=0}^\infty \frac{x^{2n}}{\cc_{2n,\alpha+1}},
\]
the first few Bernoulli-Dunkl polynomials are
\begin{align*}
  \B_0(x) &= 1, \qquad & \B_1(x) &= x, \\
  \B_2(x) &= x^2 - \frac{\alpha+1}{\alpha+2},\qquad &
  \B_3(x) &= x^3 - x,\\
  \B_4(x) &= x^4 - 2x^2 - \frac{(\alpha+4)(\alpha+1)}{(\alpha+3)(\alpha+2)},\qquad &
  \B_5(x) &= x^5 - 2\,\frac{\alpha+3}{\alpha+2}\,x^3 + \frac{\alpha+4}{\alpha+2}\,x.
\end{align*}

\begin{theorem}
The Bernoulli-Dunkl polynomials satisfy the recurrence relations
\begin{align}
\label{eq:x2nbd}
  x^{2n} = \B_{2n}(x)+(\alpha+1) \sum_{j=0}^{n-1} \binom{2n}{2j}_{\!\alpha}
  \, \frac{\B_{2j}(x)}{\alpha+2n-2j+1},\\
\nonumber
  x^{2n+1} = \B_{2n+1}(x) + (\alpha+1) \sum_{j=0}^{n-1} \binom{2n+1}{2j+1}_{\!\alpha}
  \, \frac{\B_{2j+1}(x)}{\alpha+2n-2j+1}.
\end{align}
Moreover,
\begin{equation}
\label{eq:LaBn}
  \La(\B_n) = (n+(\alpha+1/2)(1-(-1)^n)) \B_{n-1}
\end{equation}
and
\begin{enumerate}
\item $\B_{2n}$ is an even polynomial, $n\ge 0$;
\item $\B_{2n+1}$ is an odd polynomial, $n\ge 0$, and vanishes at $1$ (and hence at~$-1$) for $n\ge 1$.
\end{enumerate}
\end{theorem}

\begin{proof}
It is an easy consequence of Lemma~\ref{lem:xn}.
\end{proof}

\begin{remark}
\label{rem1}
One of the reasons why we call the family $\{\B_n\}_{n=0}^{\infty}$ Bernoulli-Dunkl polynomials is the following.
From its generating function one gets
\begin{equation}
\label{eq:BtoB}
  \frac{\B_{n,-1/2}(2x-1)}{2^n} = B_n(x),
\end{equation}
where $\{B_n\}_{n=0}^{\infty}$ are the Bernoulli polynomials (for the definition and properties of the Bernoulli polynomials one can see, for instance,~\cite{Dil-NIST}).

Indeed, taking into account that
\[
  E_{-1/2}(x) = e^x,\quad \I_{1/2}(x) = \frac{\sin(ix)}{ix},
\]
and replacing $x$ by $2x-1$, $t$ by $t/2$, and $\alpha$ by $-1/2$ in the definition~\eqref{eq:defbd} yields
\[
  \frac{e^{xt-t/2}}{\frac{\sin(it/2)}{it/2}}
  = \sum_{n=0}^\infty \frac{\B_{n,-1/2}(2x-1)}{2^n} \frac{t^n}{n!}.
\]
An easy computation gives
\[
  \frac{e^{xt-t/2}}{\frac{\sin(it/2)}{it/2}} = \frac{te^{xt}}{e^t-1},
\]
from where~\eqref{eq:BtoB} follows. We note that the change $x \mapsto 2x-1$ in~\eqref{eq:BtoB} is very natural in the Dunkl context, because
it is very much related to the reflection group $\ZZ_2$ (see~\eqref{eq:Duop}). For this group, the points $\pm1$ are essential, and thus the role of $x=0$ and $x=1$ for the classical Bernoulli polynomials must be translated to $-1$ and~$1$.

The other reason to call them Bernoulli-Dunkl is that they play in the sums~\eqref{eq:sum1}, involving the zeros of Bessel functions, the same role played by Bernoulli polynomials in the Euler sums~\eqref{eq:Euler}.
\end{remark}


One might expect that Bernoulli-Dunkl polynomials would satisfy many identities and formulas corresponding to known properties of Bernoulli polynomials. Although the scope of this paper is not to look for these identities, we display here just one of them to taste their flavour.

The Dunkl translation operator of a function $f$ is defined by
\[
  \tau_y f(x) = \sum_{n=0}^\infty \frac{y^n}{\cc_{n,\alpha}}\La^{n}f(x),
  \qquad \alpha>-1,
\]
where $\La^0$ is the identity operator and $\La^{n+1} = \La(\La^n)$. For the Dunkl transform, the translation $\tau_y$ plays the same role as the classical translation for the Fourier transform (that is, $\tau_y f(x) = f(x+y)$ for the case $\alpha=-1/2$). Some properties of the translation operator, including an integral expression, can be found in \cite{Ros}, \cite{Rosler}, and~\cite{ThX}.
For our purposes, we only need the identity \cite[formula (4.2.2)]{Ros}
\[
  \tau_y(E_\alpha(t\cdot))(x) = E_\alpha(tx) E_\alpha (ty).
\]

Then, we have the following result:

\begin{theorem}
For $\alpha>-1$, the Bernoulli-Dunkl polynomials satisfy
\begin{equation}
\label{eq:tau-B}
  \tau_y(\B_k(\cdot))(x)
  = \sum_{j=0}^k \binom{k}{j}_{\!\alpha} \B_j(x) y^{k-j}.
\end{equation}
\end{theorem}

\begin{proof}
We prove that a certain generating function is equal for both sides of \eqref{eq:tau-B}. Indeed,
\begin{align*}
  \sum_{k=0}^\infty \frac{\tau_y(\B_k(\cdot))(x)}{\cc_{k}} \, t^k
  &= \frac{1}{\I_{\alpha+1}(t)} \, \tau_y(E_{\alpha}(t\cdot))(x)
  = \frac{E_\alpha(tx)}{\I_{\alpha+1}(t)} \, E_\alpha(ty) \\
  &= \left(\sum_{k=0}^\infty \frac{\B_k(x)}{\cc_{k}} \, t^k\right)
  \left(\sum_{k=0}^\infty \frac{t^k y^k}{\cc_{k}}\right)
  = \sum_{k=0}^\infty t^k \left(\sum_{j=0}^k
  \frac{\B_j(x)}{\cc_{j}} \frac{y^{k-j}}{\cc_{k-j}}\right) \\
  &= \sum_{k=0}^\infty \frac{t^k}{\cc_{k}}
  \left(\sum_{j=0}^k \binom{k}{j}_{\!\alpha} \B_j(x)y^{k-j}\right).
\qedhere
\end{align*}
\end{proof}

In the classical case $\alpha=-1/2$, \eqref{eq:tau-B} becomes the well known translation formula
\[
  B_k(x+y) = \sum_{j=0}^k \binom{k}{j} B_j(x)y^{k-j}
\]
for Bernoulli polynomials.

\subsection{Apostol-Euler-Dunkl polynomials}

We define the Apostol-Euler-Dunkl polynomials $\{\E_{n,\alpha,u}\}_{n=0}^{\infty}$ by the generating function
\begin{equation}
\label{eq:defaed}
  \frac{u\I_{\alpha+1}(u)\Ea(xt)}{(t+u)\I_{\alpha+1}(t+u)}
 = \sum_{n=0}^\infty \frac{\E_{n,\alpha,u}(x)}{\cc_{n,\alpha}} \,t^n,
\end{equation}
where $u$ is a complex number which is neither $0$ nor a root of $\I_{\alpha+1}$.
To simplify the notation we sometimes write $\E_n=\E_{n,\alpha,u}$.

\begin{theorem}
The Apostol-Euler-Dunkl polynomials satisfy the recurrence relation
\begin{equation}
\label{eq:xned}
  x^{n} = \frac{2(1+\alpha)\cc_n}{u\I_{\alpha+1}(u)}
    \sum_{j=0}^{n} \frac{\E_{j}(x)}{\cc_j(n-j)!} \, \I_{\alpha}^{(n-j+1)}(u).
\end{equation}
Moreover,
\[
  \La(\E_n) = (n+(\alpha+1/2)(1-(-1)^n))\E_{n-1}.
\]
\end{theorem}

\begin{proof}
It is an easy consequence of Lemma~\ref{lem:xn}, taking into account that $2(1+\alpha)\I_\alpha'(x) = x\I_{\alpha+1}(x)$.
\end{proof}

The computation of Apostol-Euler-Dunkl polynomials can be simplified using the following lemma.

\begin{lemma}
\label{lem:der}
Given $\alpha\in\CC \setminus \{-1,-2,\dots \}$ and a nonnegative integer $k$, we have
\begin{equation}
\label{eq:der}
  \I_\alpha^{(k)}(z) = \I_\alpha(z) P_k(z) + \I_{\alpha+1}(z) Q_k(z),
\end{equation}
where $P_k(z)$ and $Q_k(z)$ are rational functions satisfying
$P_0(z)=1$, $Q_0(z)=0$, and the recurrence relations
\[
  P_{k+1}(z) = P_{k}'(z) + \frac{2(\alpha+1)}{z}\,Q_k(z)
\]
and
\[
  Q_{k+1}(z) = Q_k'(z) + \frac{z}{2(\alpha+1)}\,P_k(z) - \frac{2(\alpha+1)}{z}\,Q_k(z).
\]
\end{lemma}

\begin{proof}
The result is obvious for $k=0$. Let us assume that the result is true for some $k$ and prove it for $k+1$.

Let us write $\I_\alpha(z) = 2^\alpha \Gamma(\alpha+1){I_{\alpha}(z)}/{z^{\alpha}}$, where $I_\alpha$ is the modified Bessel function of the first kind and order~$\alpha$. We will use the identities (see, for instance, \cite[formulas 10.29.2]{OlMa-NIST})
\begin{equation}
\label{eq:Ider+}
  I_\alpha'(z) = I_{\alpha+1}(z) + \frac{\alpha}{z}\,I_\alpha (z)
\end{equation}
and
\begin{equation}
\label{eq:Ider-}
  I_\alpha'(z) = I_{\alpha-1}(z) - \frac{\alpha}{z}\,I_\alpha (z).
\end{equation}
By \eqref{eq:Ider+} we have
\[
  \I_\alpha'(z) = 2^\alpha \Gamma(\alpha+1) \left(\frac{I_\alpha'(z)}{z^\alpha}
  - \alpha \, \frac{I_\alpha(z)}{z^{\alpha+1}}\right)
  = \frac{z}{2(\alpha+1)} \, \I_{\alpha+1}(z),
\]
and by \eqref{eq:Ider-} (with $\alpha+1$ instead of~$\alpha$) we deduce that
\begin{align*}
  \I_{\alpha+1}'(z) &= 2^{\alpha+1} \Gamma(\alpha+2) \left(\frac{I_{\alpha+1}'(z)}{z^{\alpha+1}}
  - (\alpha+1)\frac{I_{\alpha+1}(z)}{z^{\alpha+2}}\right) \\
  &= 2^{\alpha+1} \Gamma(\alpha+2) \left(\frac{I_{\alpha}(z)}{z^{\alpha+1}}
  - 2(\alpha+1)\frac{I_{\alpha+1}(z)}{z^{\alpha+2}}\right) \\
  &= \frac{2(\alpha+1)}{z}(\I_{\alpha}(z)-\I_{\alpha+1}(z)).
\end{align*}
Then,
\begin{align*}
  \I_{\alpha}^{(k+1)}(z) &= \I_\alpha'(z)P_k(z) + \I_\alpha(z)P_k'(z)
  + \I_{\alpha+1}'(z)Q_k(z) + \I_{\alpha+1}(z)Q_k'(z) \\
  &= \frac{z}{2(\alpha+1)} \, \I_{\alpha+1}(z)P_k(z)+\I_\alpha(z)P_k'(z) \\*
  &\qquad\qquad+ \frac{2(\alpha+1)}{z}(\I_{\alpha}(z) - \I_{\alpha+1}(z))Q_k(z)
  + \I_{\alpha+1}(z)Q_k'(z) \\
  &= \I_\alpha(z) \left(P_k'(z)+\frac{2(\alpha+1)}{z}Q_k(z)\right) \\*
  &\qquad\qquad+ \I_{\alpha+1}(z) \left(Q_k'(z) + \frac{z}{2(\alpha+1)}\,P_k(z) - \frac{2(\alpha+1)}{z}\,Q_k(z)\right)
\end{align*}
and \eqref{eq:der} follows.
\end{proof}

The identities in the previous lemma become much simpler when $z = ui$ and $u$ is a zero of the Bessel function $J_\alpha$ or $J_{\alpha+1}$. Indeed,
if $s_{n,\alpha}$ is a non-null zero of $J_{\alpha+1}(x)$ and (as usual in the literature) $j_{n,\alpha}=s_{n,\alpha-1}$, then:

\begin{corollary}
\label{cor:cmn}
Given $\alpha\in\CC \setminus \{-1,-2,\cdots \}$ and a nonnegative integer $k$, we have
\[
  \I_\alpha^{(k)}(is_l) = \I_\alpha(is_l) P_k(is_l)
\]
where the first values for $P_k(z)$  are
$P_0(z)=1$, $P_1(z)=0$, $P_2(z)=1$,
\[
  P_3(z) = -\frac{2\alpha+1}{z},
  \qquad \text{ and }
  \qquad P_4(z) = \frac{(2\alpha+1)(2\alpha+3)}{z^2}+1.
\]
Moreover,
\[
  \I_\alpha^{(k)}(ij_l) = \I_{\alpha+1}(ij_l) Q_k(ij_l)
\]
where the first values for  $Q_k(z)$ are
$Q_0(z)=0$, $\displaystyle Q_1(z)=\frac{z}{2(\alpha+1)}$,
\begin{align*}
  Q_2(z) &= -\frac{2\alpha+1}{2(\alpha+1)},
  \qquad Q_3(z) = \frac{z}{2(\alpha+1)}+\frac{(2\alpha+1)}{z},
  \qquad \text{ and }\\
  Q_4(z) &= -(1+2\alpha)\left(\frac{1}{1+\alpha}+\frac{3+2\alpha}{z^2}\right).
\end{align*}
\end{corollary}

This gives the first few Apostol-Euler-Dunkl polynomials $\E_{n,\alpha,ij_l}$ as follows:
\begin{align*}
  \E_{0,\alpha,ij_l}(x) &= 1, \\
  \E_{1,\alpha,ij_l}(x) &= x + \frac{2(1+\alpha)(1+2\alpha)}{ij_l}, \\
  \E_{2,\alpha,ij_l}(x) &= x^2 + \frac{2(1+2\alpha)}{i j_l} \, x
    - 2(1+\alpha)\left(1+\frac{2\alpha(1+2\alpha)}{j_l^2}\right), \\
  \E_{3,\alpha,ij_l}(x) &= x^3 + \frac{2(2+\alpha)(1+2\alpha)}{i j_l} \,x^2
    - 2(2+\alpha)\left(1+\frac{2\alpha(1+2\alpha)}{j_l^2}\right) x \\
  &\qquad - \frac{8(2+\alpha)(1+\alpha)(1+2\alpha)}{3ij_l}
    \left(2+\frac{\alpha(-1+2\alpha)}{j_l^2}\right).
\end{align*}

The Apostol-Euler-Dunkl polynomials can be considered as a generalization of the Apostol-Euler
polynomials $\{\mathcal{E}_n(x;\lambda)\}_{n=0}^\infty$ defined by the generating function
\[
\frac{2}{\lambda e^t+1}\,e^{xt}
= \sum_{n=0}^\infty \mathcal{E}_n(x;\lambda) \frac{t^n}{n!}
\]
(see, for instance, \cite{NRV}). Indeed, one can recover the Apostol-Euler polynomials replacing $x$ by $2x-1$ and $t$ by $t/2$, and taking $\alpha=-1/2$ and $\lambda=-e^{2u}$ in the definition~\eqref{eq:defaed} of the Apostol-Euler-Dunkl polynomials:
\begin{equation}
\label{eq:EtoE}
  \frac{\E_{n,-1/2,\log(-\lambda)/2}(2x-1)}{2^{n-1}(\lambda+1)} = \mathcal{E}_n(x;\lambda).
\end{equation}
We omit the computation of~\eqref{eq:EtoE} because it is similar to the one in Remark~\ref{rem1} in the previous section.
(We do not consider any Dunkl version of the Apostol-Bernoulli polynomials \cite{Ap} because they are very close relatives of the Apostol-Euler polynomials, see \cite[Lemma~2]{NRV}.)

Taking $\lambda=1$ (or $u=i\pi/2$) in~\eqref{eq:EtoE} gives
\[
  \frac{\E_{n,-1/2,i\pi/2}(2x-1)}{2^n} = E_{n}(x),
\]
where $\{E_{n}\}_{n=0}^{\infty}$ are the Euler polynomials.

Bernoulli-Dunkl polynomials can be obtained from the Apostol-Euler-Dunkl polynomials by taking limit when $u$ goes to~$0$. Indeed, the generating function~\eqref{eq:defaed} can be written as
\[
\frac{\I_{\alpha+1}(u)\Ea(tx)}{\I_{\alpha+1}(t+u)}
= \left(1+\frac{t}{u}\right) \sum_{n=0}^\infty \frac{\E_{n,\alpha,u}(x)}{\cc_n} \,t^n.
\]
An easy computation gives
\[
  \frac{\I_{\alpha+1}(u)\Ea(tx)}{\I_{\alpha+1}(t+u)}
  = 1 + \sum_{n=1}^\infty \left(\frac{\E_{n,\alpha,u}(x)}{\cc_n}
  + \frac{\E_{n-1,\alpha,u}(x)}{u\cc_{n-1}}\right)t^n.
\]
Taking limit when $u$ goes to $0$ and using the generating function~\eqref{eq:defbd} proves that
\[
  \B_{n,\alpha}(x) = \lim_{u\to 0} \left(\E_{n,\alpha,u}(x)
  + \frac{\cc_n \E_{n-1,\alpha,u}(x)}{u\cc_{n-1}}\right).
\]

\subsection{Calogero-Dunkl numbers}

We finally introduce what we have called Calogero-Dunkl numbers; later in this paper, we will use them to sum the Calogero type series~\eqref{eq:calggi} that appear in~\cite{Cal}. The Calogero-Dunkl numbers $\{\am_{n,\alpha,u}\}_{n=0}^\infty$ are defined by the generating function
\begin{equation}
\label{eq:mamoneon}
  \frac{\I_\alpha(t+u)}{(t+u)\I_{\alpha+1}(t+u)} = \sum_{n=0}^\infty \am_{n,\alpha,u}t^n,
\end{equation}
where the parameter $u$ is neither $0$ nor a zero of the entire function $\I_{\alpha+1}$.
This gives the following recurrence for the Calogero-Dunkl numbers:
\begin{equation}
\label{eq:mamoneon2}
  \I_\alpha^{(n)}(u) = 2(1+\alpha)n! \sum_{j=0}^n \am_{j,\alpha,u} \frac{\I_\alpha^{(n-j+1)}(u)}{(n-j)!},
\end{equation}
with $\am_{0,\alpha,u}=\I_\alpha(u)/(u\I_{\alpha+1}(u))$.

Lemma~\ref{lem:der} can be used to compute the Calogero-Dunkl numbers. In particular, from Corollary~\ref{cor:cmn}
we get the first few Calogero-Dunkl numbers~$\am_{n,\alpha,ij_l}$:
\begin{align}
\label{eq:mis1}
  \am_{0,\alpha,ij_l} &=0,
  \quad &\am_{1,\alpha,ij_l} &= \frac{1}{2(1+\alpha)},\\
\label{eq:mis2}
  \am_{2,\alpha,ij_l} &= \frac{1+2\alpha}{4(1+\alpha)ij_l},
  \quad &\am_{3,\alpha,ij_l} &= -\frac{1}{6(1+\alpha)}+\frac{1-4\alpha^2}{12(1+\alpha)j_l^2}.
\end{align}

The Calogero-Dunkl numbers can be used to define the associated Calogero-Dunkl polynomials by setting
\[
  A(t)=\frac{(t+u)\I_{\alpha+1}(t+u)}{\I_\alpha (t+u)}
\]
in the definition~\eqref{eq:ad} of the Appell-Dunkl polynomials (but we do not study these polynomials here).

\section{Partial fraction decomposition for the Bessel functions}
\label{sec:simplefr}

We will use the polynomials defined in the previous section to sum some series involving the zeros of Bessel functions (see Section~\ref{sec:serieszeros}). To this end we need a couple of partial fraction decompositions for the Bessel functions.

The first one is the following. For each complex number $\alpha$, except for the negative integers,
\begin{equation}
\label{eq:pfd0}
   \frac{J_\alpha(t)}{J_{\alpha+1}(t)}
   = \frac{2(\alpha+1)}{t} + \sum_{j=1}^{\infty} \frac{2t}{t^2-s_j^2},
\end{equation}
where the series converges uniformly on compact sets of $\CC \setminus \{\pm s_1,\pm s_2,\dots\}$. The proof is essentially the same as for the well-known partial fraction decomposition of $\frac{J_{\alpha+1}(t)}{J_\alpha(t)}$
(see~\cite[\S\,15.41, p.~497--498]{Wat}). The expansion~\eqref{eq:pfd0} can be rewritten as
\begin{equation}
\label{eq:pfd1}
   \frac{\I_\alpha(it)}{t\I_{\alpha+1}(it)}
   = \frac{1}{t} +
   \frac{1}{2(\alpha+1)}
   \sum_{j\in \ZZO} \frac{1}{t-s_j}.
\end{equation}
Let us recall here that throughout this paper any doubly infinite series of the form $\sum_{j\in \ZZ}$  must be understood as the principal value, that is, the limit of $\sum_{|j|\leq N}$ as $N$ goes to infinity, and hence the series in \eqref{eq:pfd1} converges uniformly on compact sets of
$\CC \setminus \{\pm s_1,\pm s_2,\dots\}$.

We will also need a partial fraction decomposition for $\frac{1}{t\I_{\alpha+1}(it)}$.
The same argument of~\cite[\S\,15.41, p.~497--498]{Wat} produces this partial fraction decomposition. The proof goes as follows: take a large rectangle $D=D(A,B)$ with vertices $\pm A \pm i B$, where $A$ and $B$ are positive, and consider
\[
   \frac{1}{2\pi i} \int_{\partial D}
   \frac{1}{(w-t)w \I_{\alpha+1}(iw)} \, dw,
\]
$\partial D$ meaning the border of $D$.
The poles of $\frac{1}{(w-t)w \I_{\alpha+1}(iw)}$ in $D$, all of them simple, are $t$, $0$, and those $s_j \in D$. The residue at $t$ is obviously
\[
   \frac{1}{t \I_{\alpha+1}(it)};
\]
the residue at $0$ is
\[
   - \frac{1}{t \I_{\alpha+1}(0)} = - \frac{1}{t}.
\]
Finally, the residue at each $s_j \in D$ is
\[
   \lim_{w \to s_j} \frac{w-s_j}{(w-t)w \I_{\alpha+1}(iw)}
   = \frac{1}{(s_j - t) s_j \I_{\alpha+1}'(is_j)}
   = \frac{1}{(s_j - t) 2(\alpha+1) \I_{\alpha}(is_j)}.
\]
Thus, the calculus of residues gives
\begin{equation}
\label{eq:residuo}
   \frac{1}{2\pi i} \int_{\partial D}
   \frac{1}{(w-t)w \I_{\alpha+1}(iw)} \, dw
   = \frac{1}{t \I_{\alpha+1}(it)} - \frac{1}{t}
   + \sum_{s_j \in D}
   \frac{1}{(s_j - t) 2(\alpha+1) \I_{\alpha}(is_j)}.
\end{equation}
Using arguments similar to those of~\cite[\S\,15.41, p.~498]{Wat}, the values of $A$ and $B$ can be chosen arbitrarily large and such that
\begin{equation}
\label{eq:cotainf}
   |J_{\alpha+1}(w)| \geq C |w|^{-1/2}
\end{equation}
on $\partial D$ for some constant $C > 0$ independent of $A$ and~$B$.
We sketch here the proof for the sake of completeness. The starting point is the equality
$2J_{\alpha+1}(w) = H^{(1)}_{\alpha+1}(w) + H^{(2)}_{\alpha+1}(w)$, where the Bessel functions of the third kind satisfy the estimates
\begin{align}
\label{ec:estH1}
   H_{\alpha+1}^{(1)}(w) &= \left(\frac{2}{\pi w}\right)^{1/2}
   e^{i(w-\frac{1}{2}(\alpha+1)\pi - \frac{1}{4})}
   \{ 1 + \eta_{1,\alpha+1}(w)\},
   \\
\label{ec:estH2}
   H_{\alpha+1}^{(2)}(w) &= \left(\frac{2}{\pi w}\right)^{1/2}
   e^{-i(w-\frac{1}{2}(\alpha+1)\pi - \frac{1}{4})}
   \{ 1 + \eta_{2,\alpha+1}(w)\},
\end{align}
$\eta_{1,\alpha+1}(w)$ and $\eta_{2,\alpha+1}(w)$ being $\mathcal{O}(1/w)$ for large $|w|$ \cite[\S\,15.4, p.~496]{Wat}. Then, outside the horizontal strip $|\Im(w-\frac{1}{2}(\alpha+1)\pi - \frac{1}{4})| \leq \log 2$ we have either $|e^{i(w-\frac{1}{2}(\alpha+1)\pi - \frac{1}{4})}| < \frac{1}{2}$ or $|e^{i(w-\frac{1}{2}(\alpha+1)\pi - \frac{1}{4})}| > 2$, so that either $H_{\alpha+1}^{(1)}$ or $H_{\alpha+1}^{(2)}$ dominates the other and~\eqref{eq:cotainf} follows. The whole $\partial D$ is thus covered, except for two vertical lines of length $2\log 2$ with $\Re w = \pm A$. On these two pieces, according to~\eqref{ec:estH1} and~\eqref{ec:estH2} the problem reduces essentially to get a lower estimate for
$|\cos (w-\frac{1}{2}(\alpha+1)\pi - \frac{1}{4})|$,
which can be done by simply chosing $A$ so that to avoid the zeros of the cosine function.

Since $|w|^\alpha \leq C|w|^{\Re \alpha}$, it follows from~\eqref{eq:cotainf} that
\[
   \sup_{w \in \partial D} \left|
   \frac{1}{w \I_{\alpha+1}(iw)} \right|
   \leq C^{-1} |w|^{\Re\alpha+1/2}.
\]

Then, the left-hand side of~\eqref{eq:residuo} goes to $0$ as $A$ and $B$ go to infinity, provided that $\Re\alpha + \frac{1}{2} < 0$, and this proves the partial fraction decomposition
\[
   \frac{1}{t \I_{\alpha+1}(it)} = \frac{1}{t}
   + \frac{1}{2(\alpha+1)}
   \sum_{j \in \ZZO}
   \frac{1}{(t - s_j) \I_{\alpha}(is_j)},
   \qquad \text{for } \Re\alpha + \frac{1}{2} < 0.
\]
If we consider, for a fixed positive integer $m$, the integral
\[
   \frac{1}{2\pi i} \int_{\partial D}
   \frac{1}{(w-t)^{m+1} w \I_{\alpha+1}(iw)} \, dw,
\]
then the same arguments work with the only difference that the integrand now has a pole of order $m+1$ at $t$, so the residue at $t$ is
\[
   \frac{1}{m!} \left.\frac{d^m}{d w^m} \left(
   \frac{1}{w \I_{\alpha+1}(iw)} \right) \right|_{w=t}.
\]
The residues at $0$ and those those $s_j \in D$ are, respectively,
\[
    \frac{1}{(-t)^{m+1}}
\]
and
\[
   \frac{1}{(s_j - t)^{m+1} 2(\alpha+1) \I_{\alpha}(is_j)}.
\]
This gives the expansion
\begin{equation}
\label{eq:pfd2}
 \left(
   \frac{1}{t \I_{\alpha+1}(it)} \right) ^{(m)}
   = \frac{(-1)^m m!}{t^{m+1}}
   + \frac{(-1)^m m!}{2(\alpha+1)}
   \sum_{j \in \ZZO}
   \frac{1}{(t - s_j)^{m+1} \I_{\alpha}(is_j)},
\end{equation}
which is valid if $\Re\alpha - m + \frac{1}{2} < 0$.

\section{Sums involving the zeros of the Bessel functions}
\label{sec:serieszeros}

In this section, we will use the Bernoulli-Dunkl, Apostol-Euler-Dunkl polynomials and the Calogero-Dunkl numbers to sum a bunch of series involving the zeros of Bessel functions.

For any complex value of $\alpha$ (except for the negative integers), we can order the zeros $s_{j,\alpha}$ of the Bessel function $J_{\alpha+1}(x)/x^{\alpha+1}$ so that $s_{j,\alpha}=-s_{-j,\alpha}$ and $0<\Re s_{j,\alpha}\le \Re s_{j+1,\alpha}$, $j\ge 0$ (\cite[\S\,15.41, p.~497]{Wat}). To simplify the notation, we write $s_j=s_{j,\alpha}$.
The case $\alpha>-2$ is particularly relevant because then $s_j$, $j\ge 0$, are positive numbers; but complex zeros appear when $\alpha<-2$ (\cite[\S\,15.27, p.~483]{Wat}). However, their imaginary parts are bounded for each (bounded set of) $\alpha$, see~\cite[\S\,15.4, p.~497]{Wat}.

Let us consider the sums (depending on the parameter $\alpha$, some of them might be divergent)
\begin{align}
\label{eq:sk}
  \sigma_k &= \sum_{j=1}^\infty \frac{1}{s_j^{2k}},\\
\label{eq:rk}
  \varrho_k &= \sum_{j=1}^\infty \frac{1}{\I_{\alpha}(is_j)s_j^{2k}},\\
\label{eq:mam}
  \eta_{k,u} &= \sum_{j\in \ZZO} \frac{1}{(s_j-u)^{k+1}},\\
\label{eq:cal}
  \eta_k^{\{l\}} &= \sum_{j\in \ZZ\setminus\{0,l\}} \frac{1}{(s_j-s_l)^{k+1}},\\
\label{eq:vku}
  \omega_{k,u} &= \sum_{j\in \ZZO}\frac{1}{\I_\alpha(is_j)(s_j-u)^{k+1}},\\
\label{eq:vk}
  \omega_k^{\{l\}} &= \sum_{j\in \ZZ\setminus\{0,l\}} \frac{1}{\I_\alpha(is_j)(s_j-s_l)^{k+1}},
\end{align}
where $u\in \CC\setminus\{\pm s_1,\pm s_2,\dots\}$.
It has been already mentioned that series like $\sum_{j\in \ZZ}$ or similar must be understood as the principal value, that is, the limit of $\sum_{|j|\leq N}$ as $N$ goes to infinity.

Notice that for $\alpha=-1/2$ we have $s_j=j\pi$ and $\I_{\alpha}(is_j)=(-1)^j$,
so we can think of $\varrho_k$, $\omega_{k,u}$, and $\omega_k^{\{l\}}$ as a kind of ``alternating'' series.

Let us start with the two first sums, which are computed using the Bernoulli-Dunkl polynomials.

\begin{theorem}
\label{thm:sumassr}
Let $\alpha$ be a complex number which is not a negative integer and $k \ge 1$. Then
\begin{equation}
\label{eq:alpm1}
  \sigma_{k} = \frac{(-1)^{k+1}}{2^{2k}k!\,(\alpha+2)_{k-1}} \,\B_{2k}(1).
\end{equation}
If, in addition, $\Re \alpha < 2k -3/2$, then
\begin{equation}
\label{eq:alpm2}
  \varrho_{k} = \frac{(-1)^{k+1}}{2^{2k}k!\,(\alpha+2)_{k-1}} \,\B_{2k}(0).
\end{equation}
\end{theorem}

\begin{proof}
Let us begin with~\eqref{eq:alpm1}. The starting point is the partial fraction decomposition~\eqref{eq:pfd1}, which can be written as
\begin{equation*}
  \frac{\I_\alpha (it)}{t\I_{\alpha+1}(it)} = \frac{1}{t}
  + \frac{1}{\alpha+1} \sum_{j=1}^\infty \frac{t}{t^2-s_j^2},
\end{equation*}
where, for $\alpha\in \CC\setminus\{-1,-2,\dots \}$, the series
converges absolutely (and uniformly on compact sets of $\CC\setminus\{\pm s_1,\pm s_2,\dots \}$).

Using the geometric series and changing the order of the sums (the absolute convergence allows the use of Fubini's theorem), we get
\begin{align}
\nonumber
  \frac{\I_\alpha (it)}{t\I_{\alpha+1}(it)}
  &= \frac{1}{t} -
  \frac{1}{\alpha+1}
  \sum_{j=1}^\infty \frac{t}{s_j^2} \sum_{n=0}^\infty \frac{t^{2n}}{s_j^{2n}}
  \\
\label{eq:alpm3}
  &= \frac{1}{t} -
  \frac{1}{\alpha+1}
  \sum_{k=1}^\infty \Big(\sum_{j=1}^\infty \frac{1}{s_j^{2k}} \Big)  t^{2k-1},
\end{align}
valid if $|t| < |s_j|$ for every~$j$.
On the other hand, evaluating the generating function \eqref{eq:defbd} for the Bernoulli-Dunkl polynomials at $x=1$ gives
\[
  \frac{\I_\alpha (t)}{\I_{\alpha+1}(t)} + \frac{t}{2(\alpha+1)}
  = \sum_{k=0}^\infty \frac{\B_{k,\alpha}(1)}{\cc_{k,\alpha}} \,t^k.
\]
Taking into account that $\B_0=1, \B_1(x)=x$ and that $\B_{2k+1}(x)$ vanishes at $x=1$ for $k\ge 1$, we get
\begin{equation}
\label{eq:alpm4}
  \frac{\I_\alpha (t)}{t\I_{\alpha+1}(t)}
  = \frac{1}{t} + \sum_{k=1}^\infty \frac{\B_{2k,\alpha}(1)}{\cc_{2k,\alpha}} \,t^{2k-1}.
\end{equation}
Now, the identity \eqref{eq:alpm1} follows easily by comparing \eqref{eq:alpm3} and~\eqref{eq:alpm4}.

The proof of \eqref{eq:alpm2} proceeds in a similar way, but using now the partial fraction decomposition~\eqref{eq:pfd2} with $m=2k-1$, i.e.,
\[
  \left(\frac{1}{t\I_{\alpha+1}(it)}\right)^{(2k-1)}
  = -\frac{(2k-1)!}{t^{2k}} - \frac{1}{2(\alpha+1)}
    \sum_{j\in\ZZO} \frac{(2k-1)!}{\I_\alpha(is_j)(t-s_j)^{2k}},
\]
where, for $\alpha\in \CC\setminus\{-1,-2,\dots \}$ with $\Re \alpha < 2k-3/2$, the series converges absolutely. Putting together the terms with $j$ and $-j$, using the power series expansion
\[
   \frac{1}{(t-s_j)^{2k}} + \frac{1}{(t+s_j)^{2k}}
   = \sum_{n \geq k}
   2 \binom{2n-1}{2k-1} \frac{t^{2n-2k}}{s_j^{2n}},
\]
valid if $|t| < |s_j|$ for every $j$, and applying Fubini's theorem to change the order of the sums give
\begin{multline}
\label{eq:pfd2b}
  \left(\frac{1}{t\I_{\alpha+1}(it)}\right)^{(2k-1)}
  = -\frac{(2k-1)!}{t^{2k}}
  \\- \frac{1}{\alpha+1}
    \sum_{n\geq k}
    \Big(\sum_{j\geq 1} \frac{1}{\I_\alpha(is_j)s_j^{2n}} \Big)
    (2n-1)(2n-2)\dots(2n+1-2k) t^{2n-2k}.
\end{multline}
Now, taking $x=0$ in~\eqref{eq:defbd} gives
\[
   \frac{1}{t\I_{\alpha+1}(it)} = \frac{1}{t} + \sum_{n\geq 1}
   \frac{\B_{2n,\alpha}(0)}{\cc_{2n,\alpha}} (-1)^n t^{2n-1}.
\]
After differentiating $(2k-1)$ times, we can compare the constant term in the resulting expansion with the constant term in~\eqref{eq:pfd2b} to obtain~\eqref{eq:alpm2}.
\end{proof}

Notice that for $\alpha=-1/2$, Theorem~\ref{thm:sumassr} gives the corresponding formulas for $\sum_{j=1}^{\infty} 1/j^{2k}$ and $\sum_{j=1}^{\infty} (-1)^j/j^{2k}$ (see~\eqref{eq:Euler} and~\eqref{eq:Euler2}, respectively) in terms of Bernoulli polynomials (remember that, by~\eqref{eq:BtoB}, $\B_{n,-1/2}(x) = 2^n B_n((x+1)/2)$).

As explained above, the series $\sigma_k$ and $\varrho_k$ have already been studied in the literature, but the expressions~\eqref{eq:alpm1} and~\eqref{eq:alpm2} in terms of the Bernoulli-Dunkl polynomials are new.

The recurrence relation~\eqref{eq:x2nbd} for the Bernoulli-Dunkl polynomials gives the corresponding recurrence relation for the sums~\eqref{eq:sk} and~\eqref{eq:rk}.

\begin{corollary}
Under the hypothesis of Theorem~\ref{thm:sumassr},
\begin{align*}
  n &= \sum_{j=1}^n (-1)^{j+1}4^j(n-j+1)_j (\alpha+n-j+2)_j \sigma_{j}, \\
  -\alpha-1 &= \sum_{j=1}^n (-1)^{j+1}4^j(n-j+1)_j(\alpha+n-j+2)_j \varrho_{j}.
\end{align*}
\end{corollary}

The first relation can be found in \cite[p.~149]{Sne}, or in
\cite[identity~(14)]{Ki}; the second one in \cite[p.~151]{Sne}.

The Calogero-Dunkl numbers~\eqref{eq:mamoneon} can  be used to sum the series~\eqref{eq:mam} and~\eqref{eq:cal}.

\begin{theorem}
\label{thm:cag}
For $\alpha \in \CC\setminus\{-1,-2,\dots\}$, $k=0,1,2,\dots$, and
$u\in \CC\setminus\{0, \pm s_1,\allowbreak\pm s_2,\dots\}$, the sum of the series $\eta_{k,u}$ defined in \eqref{eq:mam} is
\begin{equation}
\label{eq:mam1i}
  \eta_{k,u} = 2(1+\alpha) \left(\frac{(-1)^k}{u^{k+1}}-i^{k+1}\am_{k,\alpha,iu}\right).
\end{equation}
Moreover, the series $\eta_{k,u}$ \eqref{eq:mam} and $\eta_k^{\{l\}}$ \eqref{eq:cal} satisfy the recurrence relations
\begin{align}
\label{eq:mam1}
  \sum_{j=0}^k \frac{\I_\alpha^{(k-j+1)}(iu)}{i^j(k-j)!}
  \left(\eta_{j,u}-\frac{(-1)^j 2(1+\alpha)}{u^{j+1}}\right)
  &= -\frac{\I_\alpha^{(k)}(iu)}{k!}, \\
\label{eq:mam0}
  \sum_{j=0}^k \frac{\I_\alpha^{(k-j+1)}(is_l)}{i^{j+1}(k-j)!}
  \left(\eta_j^{\{l\}}-\frac{(-1)^j 2(1+\alpha)}{s_l^{j+1}}\right)
  &= \frac{\I_\alpha^{(k+2)}(is_l)}{(k+1)!} - \frac{\I_\alpha^{(k)}(is_l)}{k!}.
\end{align}
In particular,
\begin{equation}
\label{eq:cal1}
  \eta_0^{\{l\}} = \frac{3+2\alpha}{2s_l},\quad
  \eta_1^{\{l\}} = -\frac{(3+2\alpha)(7+2\alpha)}{12s_l^2}+\frac{1}{3}.
\end{equation}
\end{theorem}

\begin{proof}
The starting point is again the partial fraction decomposition \eqref{eq:pfd1}, i.e.,
\begin{equation}
\label{eq:pfd1c}
  \frac{\I_\alpha (it)}{t\I_{\alpha+1}(it)}
  = \frac{1}{t} + \frac{1}{2(\alpha+1)} \sum_{j=1}^\infty \left(\frac{1}{t-s_j}+\frac{1}{t+s_j}\right).
\end{equation}
Writing
\begin{align*}
  t-s_j &= (u-s_j)\left(1-\frac{t-u}{s_j-u}\right),
  \\
  t+s_j &= (u+s_j)\left(1-\frac{t-u}{-s_j-u}\right),
\end{align*}
inserting it in~\eqref{eq:pfd1c} and using the geometric series leads to
\[
  \frac{\I_\alpha(it)}{t\I_{\alpha+1}(it)}
  = \frac{1}{t} - \frac{1}{2(\alpha+1)} \sum_{j=1}^\infty
  \sum_{k=0}^\infty (t-u)^k \left( \frac{1}{(s_j-u)^{k+1}} + \frac{1}{(-s_j-u)^{k+1}} \right).
\]
It is easy to check that
\[
   \left| \frac{1}{(s_j-u)^{k+1}} + \frac{1}{(-s_j-u)^{k+1}} \right|
   \leq C \frac{1}{|s_j|^2}
\]
for every $k$ and $j$, with some constant $C$ depending on $u$, and then deduce that the double series above converges absolutely if $|t-u| < 1$. This allows the use of Fubini's theorem to get
\[
  \frac{\I_\alpha(it)}{t\I_{\alpha+1}(it)}
  = \frac{1}{t} - \frac{1}{2(\alpha+1)}
  \sum_{k=0}^\infty (t-u)^k \sum_{j\in \ZZO} \frac{1}{(s_j-u)^{k+1}},
\]
if $|t-u| < 1$. Changing $t$ to $t+u$ gives
\[
  \frac{\I_\alpha(i(t+u))}{(t+u)\I_{\alpha+1}(i(t+u))}
  = \frac{1}{t+u} - \frac{1}{2(\alpha+1)}
  \sum_{k=0}^\infty t^k \sum_{j\in \ZZO} \frac{1}{(s_j-u)^{k+1}}
\]
for $|t| < 1$. Comparing with the generating function \eqref{eq:mamoneon} for the Calogero-Dunkl numbers,
we straightforwardly get
\[
  i^k \am_{k,\alpha,iu} = \frac{(-1)^k}{u^{k+1}} - \frac{1}{2(\alpha+1)}
  \sum_{j\in \ZZO} \frac{1}{(s_j-u)^{k+1}},
\]
from where \eqref{eq:mam1i} follows.

The recurrence relation \eqref{eq:mamoneon2} for the Calogero-Dunkl numbers gives then the recurrence~\eqref{eq:mam1}. Now consider the sums
\[
  \eta_{k,u}^{\{l\}} = \sum_{j\in \ZZ\setminus\{0,l\}} \frac{1}{(s_j-u)^{k+1}}
  = \eta_{k,u}-\frac{1}{(s_l-u)^{k+1}}.
\]
Using \eqref{eq:mam1} we get
\begin{multline*}
\qquad
  \sum_{j=0}^k \frac{\I_\alpha^{(k-j+1)}(iu)}{i^j(k-j)!}
  \left(\eta_{j,u}^{\{l\}}-\frac{(-1)^j 2(1+\alpha)}{u^{j+1}}\right)
  \\
  = \sum_{j=0}^k \frac{\I_\alpha^{(k-j+1)}(iu)}{i^j(k-j)!}
  \frac{1}{(s_l-u)^{j+1}}-\frac{\I_\alpha^{(k)}(iu)}{k!}.
\qquad
\end{multline*}
Taking now limit when $u\to s_l$ proves \eqref{eq:mam0} after an easy computation.

Finally, the identities \eqref{eq:cal1} follow by taking $k=1$ and $k=2$ and using then Corollary~\ref{cor:cmn}.
\end{proof}

The first series in \eqref{eq:cal1} is the Calogero series~(\cite{Cal}).

Using the identities \eqref{eq:mis1} and \eqref{eq:mis2} for the first few Calogero-Dunkl numbers $\am_{n,\alpha,ij_l}$, where $j_l$ is a non-null zero of
the Bessel function $J_\alpha(x)$, from~\eqref{eq:mam1i} we get
\[
  \eta_{0,j_l} = \frac{2(1+\alpha)}{j_l},\quad \eta_{1,j_l}=1-\frac{2(1+\alpha)}{j_l^2},
  \quad \eta_{2,j_l} = \frac{1+2\alpha}{2j_l}+\frac{2(1+\alpha)}{j_l^3}.
\]

We finally use the Apostol-Euler-Dunkl polynomials~\eqref{eq:defaed} to sum the series \eqref{eq:vku} and~\eqref{eq:vk}.

\begin{theorem}
\label{thm:cag2}
For $\alpha \in \CC\setminus\{-1,-2,\dots\}$, $k=0,1,2,\dots$ with $\Re \alpha < k -1/2$,
and $u\in \CC\setminus\{0, \pm s_1,\pm s_2,\dots\}$, the sum of the series $\omega_{k,u}$ defined in~\eqref{eq:vku} is
\begin{equation}
\label{eq:sumAED}
  \omega_{k,u}
  = 2(1+\alpha) \left(\frac{(-1)^k}{u^{k+1}}
  - \frac{i^k\E_{k,\alpha,iu}(0)}{u\I_{\alpha+1}(iu)\cc_k}\right).
\end{equation}
Moreover, the series $\omega_{k,u}$ \eqref{eq:vku} and $\omega_k^{\{l\}}$ \eqref{eq:vk} satisfy the recurrence relations
\begin{align}
\label{eq:mamp}
  \sum_{j=0}^k \frac{\I_\alpha^{(k-j+1)}(iu)}{i^j(k-j)!}
  \left(\omega_{j,u}-\frac{2(1+\alpha)(-1)^j}{u^{j+1}}\right) &= 0,
  \\
\label{eq:mam5}
  \sum_{j=0}^k \frac{\I_\alpha^{(k-j+1)}(is_l)}{i^{j+1}(k-j)!}
  \left(\omega_j^{\{l\}}-\frac{2(1+\alpha)(-1)^j}{s_l^{j+1}}\right)
  &= \frac{\I_\alpha^{(k+2)}(is_l)}{(k+1)!\,\I_\alpha(is_l)}.
\end{align}
In particular,
\begin{align}
\label{eq:acal1}
  \omega_0^{\{l\}} &= \frac{2(1+\alpha)}{s_l} - \frac{1+2\alpha}{2s_l\I_\alpha(is_l)},\\
\label{eq:acal2}
  \omega_1^{\{l\}} &= -\frac{2(1+\alpha)}{s_l^2} -\frac{1}{6\I_\alpha(is_l)}\left( 1+\frac{(1+2\alpha)(2\alpha-3)}{2s_l^2}\right).
\end{align}
\end{theorem}

\begin{proof}
We proceed as in the proof of Theorem~\ref{thm:cag}, but using now the partial fraction decomposition~\eqref{eq:pfd2}, i.e.,
\begin{equation*}
  \left(\frac{1}{t\I_{\alpha+1}(it)}\right)^{(k)}
  = \frac{(-1)^kk!}{t^{k+1}} + \frac{1}{2(\alpha+1)} \sum_{j\in\ZZO}
  \frac{(-1)^kk!}{\I_\alpha(is_j)(t-s_j)^{k+1}},
\end{equation*}
for $\alpha\in \CC\setminus\{-1,-2,\dots\}$ with $\Re \alpha < k-1/2$.
Using again that
\[
  t-s_j = (u-s_j)\left(1-\frac{t-u}{s_j-u}\right),
\]
the $k$-th derivative of the geometric series, Fubini's theorem to permute the order of the sums, changing $t$ to $t+u$ and, finally, comparing with the $k$-th derivative of the generating function \eqref{eq:defaed} for the Apostol-Euler-Dunkl polynomials evaluated at $x=0$, we get~\eqref{eq:sumAED}.

The recurrence~\eqref{eq:xned} for the polynomials $\E_{k,\alpha,iu}$ gives for $k\ge 1$ the recurrence~\eqref{eq:mamp} for the sums~$\omega_{k,u}$.

Consider now the sums
\begin{equation}
\label{eq:vkul}
  \omega_{k,u}^{\{l\}} = \sum_{j\in \ZZ\setminus \{0,l\}} \frac{1}{\I_\alpha(is_j)(s_j-u)^{k+1}}.
\end{equation}
Since
\[
  \omega_{k,u} = \omega_{k,u}^{\{l\}} + \frac{1}{\I_\alpha(is_l)(s_l-u)^{k+1}},
\]
we get the recurrence relation
\begin{equation}
\label{eq:rr}
  \sum_{j=0}^k \frac{\I_\alpha^{(k-j+1)}(iu)}{i^j(k-j)!}
  \left(\omega_{j,u}^{\{l\}}-\frac{2(1+\alpha)(-1)^j}{u^{j+1}}\right)
  = \frac{-1}{\I_\alpha(is_l)} \sum_{j=0}^k \frac{\I_\alpha^{(k-j+1)}(iu)}{(k-j)!(s_l-u)^{j+1}}.
\end{equation}
From \eqref{eq:vkul} and \eqref{eq:vk}, one straightforwardly gets $\lim_{u\to s_l} \omega_{k,u}^{\{l\}} = \omega_k^{\{l\}}$.
Taking now limit when $u$ goes to $s_l$ in~\eqref{eq:rr} proves that
\[
  \sum_{j=0}^k \frac{\I_\alpha^{(k-j+1)}(is_l)}{i^{j+1}(k-j)!}
  \left(\omega_j^{\{l\}}-\frac{2(1+\alpha)(-1)^j}{s_l^{j+1}}\right)
  = \frac{\I_\alpha^{(k+2)}(is_l)}{(k+1)!\,\I_\alpha(is_l)},
\]
which is the identity~\eqref{eq:mam5}.

Finally, the identities \eqref{eq:acal1} and \eqref{eq:acal2} follow by taking $k=1$ and $k=2$ and using Corollary~\ref{cor:cmn}.
\end{proof}

For $\alpha=-1/2$ and $u=\pi/2$, the identity \eqref{eq:sumAED} reduces to the Euler identity~\eqref{eq:Euler2}.


\section{Fourier-Dunkl series for the Bernoulli-Dunkl polynomials}

Along this section, we assume that $\alpha>-1$. Hence,  the Bessel function $J_{\alpha+1}(x)$ has an
increasing sequence of positive zeros $\{s_j\}_{j\ge1}$, and the real function
$\Im(\Ea(ix)) = \frac{x}{2(\alpha+1)}\,\I_{\alpha+1}(ix)$
is odd and has an infinite sequence of zeros $\{s_j\}_{j\in \ZZ}$
(with $s_{-j} = -s_j$ and $s_0 = 0$).
We associate to each $\alpha$ an orthonormal system called Fourier-Dunkl system and compute the corresponding Fourier expansions of the Bernoulli-Dunkl polynomials.

\subsection{The Fourier-Dunkl orthogonal system}

The Fourier-Dunkl orthonormal system  is formed by the functions
\begin{equation}
\label{eq:eaj}
  e_{\alpha,j}(r) = \frac{2^{\alpha/2}\Gamma(\alpha+1)^{1/2}}
    {|\I_{\alpha}(is_j)|} \Ea(is_j r),
  \quad j\in \ZZO, \quad r \in (-1,1),
\end{equation}
and $e_{\alpha,0}(r) = 2^{(\alpha+1)/2}\Gamma(\alpha+2)^{1/2}$ (notice the difference of a constant factor $(2\alpha+2)^{1/2}$ with respect to~\eqref{eq:eaj}).
To simplify the notation we sometimes write $e_{\alpha,j} = e_j$.

From the definition, and taking into account that
$\Ea(0) = 1$, $\overline{\Ea(ix)} = \Ea(-ix)$,
and
$\Lambda_{\alpha,x} \Ea(\lambda x) = \lambda \Ea(\lambda x)$,
it is easy to prove the following:
\[
  e_{\alpha,0}(1) = e_{\alpha,0}(-1) = 2^{(\alpha+1)/2}\Gamma(\alpha+2)^{1/2},
\]
\begin{equation}
\label{eq:ej(1)}
  e_{\alpha,j}(1) = e_{\alpha,j}(-1) = (-1)^j \cdot 2^{\alpha/2}\Gamma(\alpha+1)^{1/2},
  \quad j \in \ZZO,
\end{equation}
\[
  \overline{e_{\alpha,j}(x)} = e_{\alpha,-j}(x),
\]
\begin{equation}
\label{eq:Lambdaej}
  \La e_{\alpha,j}(x) = is_j e_{\alpha,j}(x)
\end{equation}
(the last identity is trivial for $j=0$ because $\La$ vanishes on constant functions and $s_0=0$); notice that the factor $(-1)^j$ in~\eqref{eq:ej(1)} follows from this fact: if $r_{\nu,1},r_{\nu,2},\dots$ are the positive zeros of $J_\nu(x)$ (all of them are simple) arranged in ascending order of magnitude, then, if $\nu>-1$,
\[
  0 < r_{\nu,1} < r_{\nu+1,1} < r_{\nu,2} < r_{\nu+1,2} < \cdots,
\]
a result that is usually expressed by saying that the positive zeros of $J_\nu(x)$ interlace with those of $J_{\nu+1}(x)$ (see, for instance, \cite[\S\,15.22, p.~479]{Wat}).

Evaluating \eqref{eq:eaj} on $r=0$, we also have
\[
  e_{\alpha,j}(0) = \frac{2^{\alpha/2}\Gamma(\alpha+1)^{1/2}}{|\I_{\alpha}(is_j)|},
  \quad (-1)^j e_{\alpha,j}(0) = \frac{2^{\alpha/2}\Gamma(\alpha+1)^{1/2}}{\I_{\alpha}(is_j)},
  \qquad j \in \ZZO.
\]

With this notation and
\[
  d\mua(x) = (2^{\alpha+1}\Gamma(\alpha+1))^{-1}|x|^{2\alpha+1}\, dx,
\]
the following result was proved in~\cite{CV}:

\begin{theorem}
Let $\alpha > -1$. Then, the sequence of functions
$\{e_{\alpha,j}\}_{j\in \ZZ}$
is a complete orthonormal system in $L^2((-1,1),d\mua)$.
\end{theorem}

The $L^p$-convergence of these series was studied in~\cite{CPRV}.

The case $\alpha = -1/2$ corresponds to the classical trigonometric Fourier setting: $\I_{-1/2}(z) = \cos(iz)$, $\I_{1/2}(z) = \frac{\sin(iz)}{iz}$, $s_j = \pi j$, $E_{-1/2}(is_jx) = e^{i\pi jx}$, and $\{e_j\}_{j \in \ZZ}$ is the trigonometric system with the appropriate multiplicative constant so that it is orthonormal on $(-1,1)$ with respect to the normalized Lebesgue measure $d\mu_{-1/2}(x) = (2\pi)^{-1/2}\,dx$.

An important property that will be very useful in what follows is the following (the proof is straightforward, so it is omitted).

\begin{lemma}
Let $f$ and $g$ be two differentiable functions on the interval $[-1,1]$. Then,
\begin{equation}
\label{eq:adjoint}
  \int_{-1}^1 \La f(x) g(x) \,d\mua(x)
  = \frac{f(1)g(1) - f(-1)g(-1)}{2^{\alpha+1}\Gamma(\alpha+1)}
  - \int_{-1}^1 f(x) \La g(x) \,d\mua(x).
\end{equation}
\end{lemma}

In our context, the identity \eqref{eq:adjoint} plays the role of integration by parts.

\subsection{The Fourier-Dunkl series for the Bernoulli-Dunkl polynomials}

Let $\alpha > -1$ and $n\ge 1$ be fixed. Our aim now is to compute the Fourier-Dunkl expansion of the Bernoulli-Dunkl polynomial $\B_{n,\alpha}(x)$ with respect to the orthonormal system $\{e_{\alpha,j}\}_{j=-\infty}^{\infty}$, i.e.,
\[
  \B_n(x) = \sum_{j\in \ZZ} c_{j}(\B_n) \, e_j(x),
  \qquad
  c_{j}(\B_n) = \int_{-1}^{1} \B_n(y) \overline{e_j(y)} \,d\mua(y),
\]
where, as usual, we are dropping $\alpha$ in the notation. As we will see in the proof of the next lemma, the key facts for computing the Fourier-Dunkl coefficients $c_{j}(\B_n)$ are the properties of the Bernoulli-Dunkl polynomials (namely \eqref{eq:LaBn}, their parity and their values in~$\pm1$) as well as the formula~\eqref{eq:adjoint}.

To simplify the notation, let us rewrite \eqref{eq:LaBn} as
\begin{equation}
\label{eq:LaBnk}
  \La(\B_{n}) = (n+(\alpha+1/2)(1-(-1)^{n})) \B_{n-1}
  =: k_{n} \B_{n-1},
\end{equation}
and remember that ${\cc_{n}}/{\cc_{n-1}} = k_{n}$ (see~\eqref{eq:coccc}).
Moreover, instead of computing $c_{j}(\B_n)$ we are computing $c_{-j}(\B_n)$ because, using that $\overline{e_{-j}(y)} = e_{j}(y)$, the formulas become cleaner.

\begin{lemma}
For $j \in \ZZ$ and $n=1,2,\dots$, let us denote
\[
  c_{-j}(\B_n) = \int_{-1}^{1} \B_n(y) e_j(y) \,d\mua(y).
\]
Then,
\begin{equation}
\label{eq:c0Bn}
  c_{0}(\B_n) = 0.
\end{equation}
For $j \ne 0$,
\begin{equation}
\label{eq:cjB1}
  c_{-j}(\B_1) = \frac{-i}{s_j} \frac{(-1)^j}{2^{\alpha/2}\Gamma(\alpha+1)^{1/2}},
\end{equation}
and the recurrence relation
\[
  c_{-j}(\B_n) = \frac{k_{n}i}{s_j} \, c_{-j}(\B_{n-1}),
  \qquad n \ge 2
\]
holds. Therefore,
\[
  c_{-j}(\B_n) = \frac{-i^n \cc_n}{s_j^n}
  \frac{(-1)^j}{2^{1+\alpha/2}(\alpha+1)\Gamma(\alpha+1)^{1/2}},
  \qquad \text{for $j \ne 0$}.
\]
\end{lemma}

\begin{proof}
Using~\eqref{eq:adjoint} we have
\begin{align*}
  c_0(\B_n) &= \int_{-1}^{1} \B_n(y)e_{0}(y)\,d\mua(y)
  \\
  &= \frac{1}{k_{n+1}} \int_{-1}^1 \La \B_{n+1}(x)
  2^{(\alpha+1)/2}\Gamma(\alpha+2)^{1/2} \,d\mua(y)
  \\
  &= 2^{(\alpha+1)/2}\Gamma(\alpha+2)^{1/2} \,
  \frac{\B_{n+1}(1) - \B_{n+1}(-1)}{2^{\alpha+1}\Gamma(\alpha+1)k_{n+1}}
  \\*
  &\qquad\qquad - \frac{2^{(\alpha+1)/2}\Gamma(\alpha+2)^{1/2}}{k_{n+1}}
  \int_{-1}^1 \B_{n+1}(y) \La 1 \,d\mua(y)
  \\
  &= \frac{\Gamma(\alpha+2)^{1/2}}{2^{(\alpha+1)/2}\Gamma(\alpha+1)}
  \cdot \frac{\B_{n+1}(1) - \B_{n+1}(-1)}{k_{n+1}}.
\end{align*}
Since $\B_{n+1}(x)$ is odd if $n$ is even, and even if $n$ is odd, we have
\begin{equation*}
  c_0(\B_n) =
  \begin{cases}
  0, & \text{if $n$ is odd,}
  \\
  \dfrac{\Gamma(\alpha+2)^{1/2}}{2^{(\alpha-1)/2}\Gamma(\alpha+1)}
  \cdot \dfrac{\B_{n+1}(1)}{k_{n+1}}, & \text{if $n$ is even.}
  \end{cases}
\end{equation*}
But $\B_{2m+1}(1) = 0$ for $m\ge1$, so we obtain~\eqref{eq:c0Bn}.

To prove~\eqref{eq:cjB1}, let us apply \eqref{eq:Lambdaej}, \eqref{eq:adjoint} and~\eqref{eq:LaBnk}. Thus,
\begin{align*}
  c_{-j}(\B_{1}) &= \int_{-1}^{1} \B_{1}(x) e_{j}(y)\, d\mua(y)
  = \frac{1}{is_j} \int_{-1}^{1} \B_{1}(x) \La e_{j}(y)\, d\mua(y) \\
  &= \frac{1}{is_j}
  \frac{\B_1(1) e_{j}(1) - \B_1(-1) e_{j}(-1)}{2^{\alpha+1}\Gamma(\alpha+1)}
  - \frac{1}{is_j} \int_{-1}^{1} \La \B_1(y) e_{j}(y)\, d\mua(y) \\
  &= \frac{1}{is_j}
  \frac{\B_1(1) e_{j}(1) - \B_1(-1) e_{j}(-1)}{2^{\alpha+1}\Gamma(\alpha+1)}
  - \frac{1}{is_j} \frac{1}{k_1} \int_{-1}^{1} \B_0(y) e_{j}(y)\, d\mua(y) \\
  &= \frac{1}{is_j}
  \frac{\B_1(1) e_{j}(1) - \B_1(-1) e_{j}(-1)}{2^{\alpha+1}\Gamma(\alpha+1)}
\end{align*}
where the last integral vanishes due to the orthogonality of $\{e_{j}\}_{j\in\ZZ}$:
\[
  \int_{-1}^{1} \B_0(y) e_{j}(y)\, d\mua(y)
  = \frac{1}{2^{(\alpha+1)/2}\Gamma(\alpha+2)^{1/2}}
  \int_{-1}^{1} e_{0}(y) e_{j}(y)\, d\mua(y) = 0.
\]
Then, using that $\B_1(1) = 1$ and $\B_1(-1) = -1$, together with \eqref{eq:ej(1)},
we get~\eqref{eq:cjB1}.

To prove the recurrence relation,
let us start looking what happens if $n$ is even, say $n=2m$.
From \eqref{eq:Lambdaej}, \eqref{eq:adjoint} and \eqref{eq:LaBnk},
and taking into account that $\B_{2m}(1) = \B_{2m}(-1)$ and $e_{j}(1) = e_{j}(-1)$, it follows that
\begin{align*}
  c_{-j}(\B_{2m}) &= \int_{-1}^{1} \B_{2m}(y) e_{j}(y)\, d\mua(y)
  = \frac{1}{is_j} \int_{-1}^{1} \B_{2m}(y) \La e_{j}(y)\, d\mua(y)
  \\
  &= \frac{1}{is_j}
    \frac{\B_{2m}(1) e_{j}(1) - \B_{2m}(-1) e_{j}(-1)}
    {2^{\alpha+1}\Gamma(\alpha+1)}
  \\*
  &\qquad\qquad
  - \frac{1}{is_j} \int_{-1}^{1} \La \B_{2m}(y) e_{j}(y)\, d\mua(y)
  \\
  &= - k_{2m}\,\frac{1}{is_j} \int_{-1}^{1} \B_{2m-1}(y) e_{j}(y)\, d\mua(y)
  = k_{2m}\,\frac{i}{s_j} c_{-j}(\B_{2m-1}).
\end{align*}
In the case $n=2m+1$, we can use \eqref{eq:Lambdaej}, \eqref{eq:adjoint} and \eqref{eq:LaBnk} again, and the fact that $\B_{2m+1}(1) = \B_{2m+1}(-1) = 0$. Then,
\begin{align*}
  c_{-j}(\B_{2m+1}) &= \int_{-1}^{1} \B_{2m+1}(y) e_{j}(y)\, d\mua(y)
  = \frac{1}{is_j} \int_{-1}^{1} \B_{2m+1}(y) \La e_{j}(y)\, d\mua(y)
  \\
  &= \frac{1}{is_j}
    \frac{\B_{2m+1}(1) e_{j}(1) - \B_{2m+1}(-1) e_{j}(-1)}
    {2^{\alpha+1}\Gamma(\alpha+1)}
  \\*
  &\qquad\qquad - \frac{1}{is_j} \int_{-1}^{1} \La \B_{2m+1}(y) e_{j}(y)\, d\mua(y)
  \\
  &= - k_{2m+1}\,\frac{1}{is_j} \int_{-1}^{1} \B_{2m}(y) e_{j}(y)\, d\mua(y)
  = k_{2m+1}\,\frac{i}{s_j} c_{-j}(\B_{2m}).
\end{align*}

Finally, the recurrence relation, together with $k_nk_{n-1} \cdots k_2 = \cc_n/\cc_1$, $\cc_1=2(\alpha+1)$ and~\eqref{eq:cjB1}, gives
\[
  c_{-j}(\B_n) = \frac{k_nk_{n-1}\cdots k_2}{s_j^{n-1}} \, i^{n-1} c_{-j}(\B_1)
  = \frac{-i^n \cc_n}{s_j^n} \frac{(-1)^j}{2^{1+\alpha/2}(\alpha+1)\Gamma(\alpha+1)^{1/2}}
\]
and the proof is finished.
\end{proof}

A direct consequence of the previous lemma is the following (remember that $s_{-j} = -s_j$):

\begin{theorem}
\label{thm:sfbd}
For every $\alpha>-1$ and $n\ge 1$,
\begin{equation}
\label{eq:sfbd}
  \B_n(x) = \frac{-(-i)^n\cc_n}{2^{1+\alpha/2}(\alpha+1)\Gamma(\alpha+1)^{1/2}}
  \sum_{j\in \ZZO} \frac{(-1)^j}{s_j^n} \, e_j(x),
\end{equation}
where the convergence is in $L^2([-1,1],d\mua)$.
\end{theorem}

Notice that Theorem~\ref{thm:sumassr} would follow also by evaluating the series \eqref{eq:sfbd} at $x=1$ and $x=0$, provided that this series were proved to converge pointwisely at those points. However, the pointwise behaviour of this Fourier series is out of the scope of this paper.

In the case $\alpha=-1/2$, Theorem~\ref{thm:sfbd} becomes the Hurwitz expansion
\[
  B_n(x) = -\frac{n!}{(2\pi i)^n} \sum_{j\in \ZZO} \frac{e^{2\pi i j x}}{j^n},
  \qquad n \ge 1,
\]
for the classical Bernoulli polynomials.

For the sake of completeness, we include here other relevant Fourier-Dunkl series.

\begin{theorem}
For every $\alpha>-1$ and $t\in \CC\setminus \{0,\pm s_1,\pm s_2,\dots\}$,
\begin{equation}
\label{eq:sfeatx}
  2(\alpha+1)2^{\alpha/2}\Gamma(\alpha+1)^{1/2} \frac{\Ea(itx)}{t\I_{\alpha+1}(it)}
  = \frac{\sqrt{2(\alpha+1)}}{t}e_0
  + \sum_{j\in \ZZO} \frac{(-1)^je_j(x)}{t-s_j},
\end{equation}
where the convergence is in $L^2([-1,1],d\mua)$.
\end{theorem}

\begin{proof}
The proof is a consequence of the identity
\begin{equation}
\label{eq:BCV}
  \int_{-1}^1 \Ea(ixr)\Ea(-iyr)\, d\mua(r)
  = \frac{1}{2^{\alpha+1}\Gamma(\alpha+1)}
    \frac{\Ea(ix)\Ea(-iy)-\Ea(-ix)\Ea(iy)}{i(x-y)}
\end{equation}
(see \cite[Lemma~1]{BCV}), which holds for $\alpha > -1$, $x,y\in \mathbb{C}$, and $x \ne  y$ (the proof in \cite{BCV} is given for $\alpha\ge -1/2$ and $x,y\in \RR$ but the result can be extended to $\alpha > -1$ and $x,y\in \CC$ without any problem).

By \eqref{eq:BCV}, the Fourier-Dunkl coefficients of the function $\Ea(itx)$ are,  for $j\ne0$,
\begin{align*}
  c_j(\Ea(itx)) &= \int_{-1}^1 \Ea(itx) \overline{e_j(x)}\,d\mua(x) \\
  &= \frac{2^{\alpha/2}\Gamma(\alpha+1)^{1/2}}{|\I_\alpha(is_j)|} \int_{-1}^1 \Ea(itx) \Ea(-is_j x)\, d\mua(x) \\
  &= \frac{1}{2^{\alpha/2+1}\Gamma(\alpha+1)^{1/2}|\I_\alpha(is_j)|}
  \frac{\Ea(it)\Ea(-isj)-\Ea(-it)\Ea(is_j)}{i(t-s_j)}.
\end{align*}
Now, from the definition of $\Ea$, we have
\[
\Ea(it)\Ea(-isj)-\Ea(-it)\Ea(is_j) = \frac{it}{\alpha+1}\I_{\alpha+1}(it)
\I_{\alpha}(is_j)
\]
and
\begin{align}
  c_j(\Ea(itx)) &= \frac{\I_\alpha(is_j)}{2^{\alpha/2+1}(\alpha+1)\Gamma(\alpha+1)^{1/2}|\I_\alpha(is_j)|}
  \frac{t\I_{\alpha+1}(it)}{(t-s_j)} \notag \\
  &= \frac{(-1)^j}{2^{\alpha/2+1}(\alpha+1)\Gamma(\alpha+1)^{1/2}}
  \frac{t\I_{\alpha+1}(it)}{(t-s_j)}.
\label{eq:cjE}
\end{align}
For $j=0$, using again \eqref{eq:BCV} and the identity $\Ea(0)=1$ gives
\begin{align}
  c_0(\Ea(ixt)) &= 2^{(\alpha+1)/2}\Gamma(\alpha+2)^{1/2}\int_{-1}^1 \Ea(itx)\,d\mua(x) \notag \\
  &= \frac{\sqrt{2(\alpha+1)}}{2^{\alpha/2+1}\Gamma(\alpha+1)^{1/2}}\frac{\Ea(it)-\Ea(-it)}{it} \notag \\
  &= \frac{\sqrt{2(\alpha+1)}}{2^{\alpha/2+1}(\alpha+1)\Gamma(\alpha+1)^{1/2}} \, \I_{\alpha+1}(it).
\label{eq:c0E}
\end{align}
Finally, from \eqref{eq:cjE} and \eqref{eq:c0E}, we conclude~\eqref{eq:sfeatx}.
\end{proof}

The partial fraction decompositions \eqref{eq:pfd1} and~\eqref{eq:pfd2} (for $m=0$) are very much related to the Fourier-Dunkl series~\eqref{eq:sfeatx}. Indeed, \eqref{eq:pfd2} (for $m=0$) would follow evaluating the Fourier-Dunkl series at $x=0$. At $x=1$, the Fourier-Dunkl series \eqref{eq:sfeatx} converges to
\[
  \frac{2(\alpha+1) 2^{\alpha/2}\Gamma(\alpha+1)^{1/2}}{t\I_{\alpha+1}(it)}
  \frac{\Ea(it)+\Ea(-it)}{2}
  = 2(\alpha+1)2^{\alpha/2}\Gamma(\alpha+1)^{1/2} \frac{\I_{\alpha}(it)}{t\I_{\alpha+1}(it)}
\]
(the mean value of side limits), which gives \eqref{eq:pfd1}. But, as commented above, the pointwise convergence of the Fourier-Dunkl series \eqref{eq:sfeatx} is out of the scope of this paper.




\begin{thebibliography}{00}

\bibitem{AlS}
W. A. Al-Salam,
Characterization theorems for orthogonal polynomials,
\textit{Orthogonal Polynomials: Theory and Practice}
(Columbus, OH, 1989, P.~Nevai, ed.), 1--24,
NATO Adv. Sci. Inst. Ser. C Math. Phys. Sci., 294,
Kluwer Acad. Publ., Dordrecht, 1990.

\bibitem{Ap}
T. M. Apostol,
On the Lerch zeta function,
\textit{Pacific J. Math.} \textbf{1} (1951), 161--167;
addendum, \textit{Pacific J. Math.} \textbf{2} (1952), 10.

\bibitem{BMPS}
\'A. Baricz, D. J. Ma\v{s}irevi\'c, T. K. Pog\'any and R. Sz\'asz,
On an identity for zeros of Bessel functions,
\textit{J. Math. Anal. Appl.}
\textbf{422} (2015), 27--36.

\bibitem{BG}
Y. Ben Cheikh and M. Gaied,
Dunkl-Appell $d$-orthogonal polynomials,
\textit{Integral Transforms Spec. Funct.}
\textbf{18} (2007), 581--597.

\bibitem{BCV}
J. J. Betancor, \'O. Ciaurri and J. L. Varona,
The multiplier of the interval $[-1,1]$ for the Dunkl transform on the real line,
\textit{J. Funct. Anal.}
\textbf{242} (2007), 327--336.

\bibitem{BKI}
A. Bouanani, L. Kh\'eriji and M. Ihsen Tounsi,
Characterization of $q$-Dunkl Appell symmetric orthogonal $q$-polynomials,
\textit{Expo. Math.}
\textbf{28} (2010), 325--336.

\bibitem{Cal}
F. Calogero,
On the zeros of Bessel functions,
\textit{Lett. Nuovo Cimento (2)}
\textbf{20} (1977), 254--256.

\bibitem{CPRV}
\'O. Ciaurri, M. P\'erez, J. M. Reyes and J. L. Varona,
Mean convergence of Fourier-Dunkl series,
\textit{J. Math. Anal. Appl.}
\textbf{372} (2010), 470--485.

\bibitem{CV}
\'O. Ciaurri and J. L. Varona,
A Whittaker-Shannon-Kotel'nikov sampling theorem related to the Dunkl transform,
\textit{Proc. Amer. Math. Soc.}
\textbf{135} (2007), 2939--2947.

\bibitem{dL}
J. L. deLyra,
{On the sums of inverse even powers of zeros of regular Bessel functions},
arXiv:1305.0228 [math-ph].

\bibitem{Dil-NIST}
K. Dilcher,
Bernoulli and Euler Polynomials,
\textit{NIST handbook of mathematical functions}
(edited by F.~W. F.~Olver, D.~W. Lozier, R.~F. Boisvert and C.~W. Clark),
587--599, 
National Institute of Standards and Technology, Washington, DC, and
Cambridge University Press, Cambridge, 2010.
Available online in \url{http://dlmf.nist.gov/24}

\bibitem{DH}
I. H. Dimovski and V. Z. Hristov,
Nonlocal operational calculi for Dunkl operators,
\textit{SIGMA Symmetry Integrability Geom. Methods Appl.}
\textbf{5} (2009), Paper~030, 16~pp.

\bibitem{Du}
C. F. Dunkl,
{Differential-difference operators associated to reflection groups},
\textit{Trans. Amer. Math. Soc.}
\textbf{311} (1989), 167--183.

\bibitem{Ke}
M. K. Kerimov,
Some remarks concerning papers on the summation of series with inverse powers of zeros of Bessel functions of the first kind (Russian),
\textit{Zh. Vychisl. Mat. Mat. Fiz.}
\textbf{47} (2007), 186--188;
translation in \textit{Comput. Math. Math. Phys.}
\textbf{47} (2007), 180--182.

\bibitem{Ki}
N. Kishore,
The Rayleigh function,
\textit{Proc. Amer. Math. Soc.}
\textbf{14} (1963), 527--533.

\bibitem{Le} N. N. Lebedev,
\textit{Special functions and their applications},
Dover, New York, 1972.

\bibitem{NRV}
L. M. Navas, F. J. Ruiz and J. L. Varona,
Asymptotic estimates for Apostol-Bernoulli and Apostol-Euler polynomials,
\textit{Math. Comp.}
\textbf{81} (2012), 1707--1722.

\bibitem{OlMa-NIST}
F. W. J. Olver and L. C. Maximon,
Bessel Functions,
\textit{NIST handbook of mathematical functions}
(edited by F.~W. F.~Olver, D.~W. Lozier, R.~F. Boisvert and C.~W. Clark),
215--286, 
National Institute of Standards and Technology, Washington, DC, and
Cambridge University Press, Cambridge, 2010.
Available online in \url{http://dlmf.nist.gov/10}

\bibitem{PBM}
A. P. Prudnikov, A. Y. Brychkov and O. I. Marichev,
\textit{Integrals and series. Vol.~2: Special functions},
Gordon and Breach Science Publishers, New York, 1986.

\bibitem{Ray}
Lord Rayleigh (J. W. Strutt),
{Note on the numerical calculation of the roots of fluctuating functions},
\textit{Proc. London Math. Soc.}
\textbf{5} (1873), 119--124.

\bibitem{Ros}
M. Rosenblum,
{Generalized Hermite polynomials and the Bose-like oscillator calculus},
\textit{Oper. Theory Adv. Appl.}
\textbf{73} (1994), 369--396.

\bibitem{Rosler}
M. R\"osler,
A positive radial product formula for the Dunkl kernel,
\textit{Trans. Amer. Math. Soc.}
\textbf{355} (2003), 2413--2438.

\bibitem{Sne}
I. N. Sneddon,
On some infinite series involving the zeros of Bessel functions of the first kind,
\textit{Proc. Glasgow Math. Assoc.}
\textbf{4} (1960), 144--156.

\bibitem{ThX}
S. Thangavelu and Y. Xu,
Riesz transform and Riesz potentials for Dunkl transform,
\textit{J. Comput. Appl. Math.}
\textbf{199} (2007), 181--195.

\bibitem{Wat}
G. N. Watson,
\textit{A Treatise on the Theory of Bessel Functions} (2nd edition),
Cambridge Univ. Press,
Cambridge, 1944.

\end{thebibliography}
\end{document}